\documentclass[fleqn,3p,12pt,times]{elsarticle}
\usepackage[utf8]{inputenc}
\usepackage{amssymb,latexsym}
\usepackage{amsfonts}
\usepackage{amsmath}
\usepackage{enumerate}
\usepackage{bbm}
\usepackage{enumerate}
\usepackage{booktabs}
\usepackage{float}
\usepackage{amssymb}
\usepackage{amsthm}
\newtheorem{ex}{Example}
\newtheorem{thm}{Theorem}
\newtheorem{lem}{Lemma}

\newtheorem{con}{Conjecture}

\newcommand{\T}{\mathbb{P}}
\newcommand{\E}{\mathbb{E}}
\newcommand{\f}{\varphi}
\newcommand\dist{\buildrel d \over =}
\newtheorem{remark}{Remark}
\journal{Journal of \LaTeX\ Templates}

\begin{document}
\begin{frontmatter}

\title{Bi-seasonal discrete time risk model with income rate two
}

\author{Alina Alencenovi\v{c}\fnref{myfootnote1}}
\author{Andrius Grigutis\fnref{myfootnote2}}
\address{Institute of Mathematics, Vilnius University, Naugarduko 24, LT-03225 Vilnius}

\fntext[myfootnote1]{Email: alina.alencenovic@mif.stud.vu.lt}
\fntext[myfootnote2]{Email: andrius.grigutis@mif.vu.lt}




\begin{abstract}
This paper proceeds an approximate calculation of ultimate time survival probability for bi-seasonal discrete time risk model when premium rate equals two. The same model with income rate equal to one was investigated in 2014 by Damarackas and \v{S}iaulys. In general, discrete time and related risk models deal with possibility for a certain version of random walk to hit a certain threshold at least once in time. In this research, the mentioned threshold is the line $u+2t$ and random walk consists from two interchangeably occurring independent but not necessarily identically distributed random variables. Most of proved theoretical statements are illustrated via numerical calculations. Also, there are raised a couple of conjectures on a certain recurrent determinants non-vanishing.
\end{abstract}

\begin{keyword}

Discrete time risk model, bi-seasonal model, survival probability, random walk, net profit condition, maximum distribution.
\MSC[2020] 91G05, 60G50, 60J80.
\end{keyword}

\end{frontmatter}

\section{Introduction}

Modeling of large values has many interests across various nature sciences. Insurers may be concerned on large pay offs, biologists on vanishing of some population and etc. Models estimating likelihood of such events are often random walk based. In this research, we define the {\it random walk} (r.w.) as a sum of random variables (r.vs.) 
$
\sum_{i=1}^{t}Z_i, \, t\in\mathbb{N}.
$
Then, the {\it bi-seasonal discrete time risk model with a generalized premium rate} $W(t)$ is defined as follows 
\begin{align}\label{model}
W(t)=u+\kappa t-\sum_{i=1}^{t}Z_i,
\end{align}
where
\begin{itemize}
\item $t,\kappa\in\mathbb{N}$ and $u\in\mathbb{N}_0:=\mathbb{N}\cup\{0\}$,
\item $Z_{2i-1}{\buildrel d \over =}X$, $Z_{2i}{\buildrel d \over =}Y$ for all $i\in\mathbb{N}$ and $X,\,Y$ are independent integer valued non-negative random variables which may be distributed differently. 
\end{itemize}

Modeling insurers or other individuals wealth by (\ref{model}), the parameter $u\in\mathbb{N}_0$ is deemed as initial savings, $\kappa\in\mathbb{N}$ as income or premium rate per unit of time and r.vs. $Z_i$ are considered as occurring random claim amounts. Such type of models are discrete versions of more general Andersen model \cite{Andersen} and are widely studied all across the world. In fact, the study is heavily related to the maximum distribution of partial sums of random variables and a great initial sources on that are \cite{Spitzer}, \cite{Spitzer1} and \cite{Feller}. Scrolling across the timeline, an observable works of Gerber and Shiu on the risk collective models could be highlighted: \cite{Gerber}, \cite{Gerber1}, \cite{Shiu} and \cite{Shiu1}. Recently, many research papers on the related risk models as in (\ref{model}) are occurring per year, see for example \cite{rvz}, \cite{asmussen}, \cite{Dickson}, \cite{risks}, \cite{KS}, \cite{Scandin}, \cite{KzS} and references therein. 

The main concern of some processes, modeled by (\ref{model}), is whether its deterministic part $u+\kappa t$ is greater than random part $\sum_{i=1}^{t}Z_i$ for all natural $t$ up to some $T\in\mathbb{N}$ or even when $T\to\infty$. In some financial context, that is to know the likelihood whether initial savings and earnings are always sufficient to cover incurred expenses. More precisely, we are interested to calculate the probabilities 

\begin{align*}
&\f(u,T):=\T\left(\bigcap_{t=1}^{T}\left\{W(t)>0\right\}\right)=
\mathbb{P}\left(\max_{1\leqslant t\leqslant T}\left\{\sum_{i=1}^{t}(Z_i-\kappa)\right\}<u\right),\\ &\f(u):=\T\left(\bigcap_{t=1}^{\infty}\{W(t)>0\}\right)
=\mathbb{P}\left(\max_{t\geqslant 1}\left\{\sum_{i=1}^{t}(Z_i-\kappa)\right\}<u\right)
,  
\end{align*}
where the first one $\f(u,T)$ is called the {\it finite time survival probability} and the later $\f(u)$ -- {\it ultimate time survival probability}. Algorithms for $\f(u,T)$ are a lot more simple than for $\f(u)$. Due to complexity of $\f(u),\,u\in\mathbb{N}_0$ with arbitrary natural $\kappa$, we restrict the model (\ref{model}) to $\kappa=2$, where $\kappa=1$ case was studied in \cite{DS}. Even such a little change has a significant impact on expressions of the ultimate time survival probability $\f$.

Let us demonstrate where the expressions of $\f(u),\,u\in\mathbb{N}_0$ are stemming from. First we need to introduce some notations. For $u\in\mathbb{N}_0$ we denote
\begin{align*}
&x_u:=\T(X=u),\, y_u:=\T(Y=u),\, s_u:=\T(X+Y=u),\\
&X(u):=\sum_{i=0}^{u}x_i, \,Y(u):=\sum_{i=0}^{u}y_i,\, S(u):=\sum_{i=0}^{u}s_i,\\
&\overline{X}(u):=1-X(u), \, \overline{Y}(u):=1-Y(u),\, \overline{S}(u):=1-S(u).
\end{align*}

By the law of total probability $\mathbb{P}(A\cap B)=\mathbb{P}(A)-\mathbb{P}(A\cap B^c)$, rearrangements and other techniques from elementary probability theory, for the bi-seasonal discrete time risk model with income rate two, we get 
\begin{align*}
&\varphi(u)
=\mathbb{P}\left(\bigcap_{t=1}^{\infty}\left\{u+2 t-\sum_{i=1}^{t}Z_i>0\right\}\right)
=\mathbb{P}\left(\bigcap_{t=2}^{\infty}\left\{u+2 t-\sum_{i=1}^{t}Z_i>0\right\},\,Z_1<u+2\right)
\\ 
&=\mathbb{P}\left(\bigcap_{t=2}^{\infty}\left\{u+2 t-\sum_{i=1}^{t}Z_i>0\right\}\right)
-\mathbb{P}\left(\bigcap_{t=2}^{\infty}\left\{u+2 t-\sum_{i=1}^{t}Z_i>0\right\},\,Z_1\geqslant u+2\right)
\\
&=\mathbb{P}\left(\bigcap_{t=3}^{\infty}\left\{u+2 t-Z_1-Z_2-\sum_{i=3}^{t}Z_i>0\right\},\,Z_1+Z_2<u+4\right)\\
&-\mathbb{P}\left(\bigcap_{t=3}^{\infty}\left\{u+2 t-Z_1-Z_2-\sum_{i=3}^{t}Z_i>0\right\},\,Z_1\geqslant u+2,\,Z_1+Z_2<u+4\right)\\
&=\sum_{k=0}^{u+3}\mathbb{P}\left(\bigcap_{t-2=1}^{\infty}\left\{u+2 (t-2)+4 -Z_1-Z_2-\sum_{i=1}^{t-2}Z_i>0\right\},\,Z_1+Z_2=k\right)\\
&-(x_{u+3}y_0+x_{u+2}y_1)\T\left(\bigcap_{t=1}^{\infty}\left\{1+2 t-\sum_{i=1}^{t}Z_i>0\right\}\right)
-x_{u+2}y_0\T\left(\bigcap_{t=1}^{\infty}\left\{2+2 t-\sum_{i=1}^{t}Z_i>0\right\}\right)
\\
&=\sum_{k=0}^{u+3}\varphi(u+4-k)s_k-(x_{u+3}y_0+x_{u+2}y_1)\varphi(1)-x_{u+2}y_0\varphi(2)\\
&=\sum_{k=1}^{u+4}\varphi(k)s_{u+4-k}
-(x_{u+3}y_0+x_{u+2}y_1)\varphi(1)-x_{u+2}y_0\varphi(2).
\end{align*}

The obtained equation
\begin{align}\label{main_eq}
\f(u)=\sum_{k=1}^{u+4}\varphi(k)s_{u+4-k}-(x_{u+3}y_0+x_{u+2}y_1)\varphi(1)-x_{u+2}y_0\varphi(2)
\end{align}
is the core point of this research. By setting $u=0,1,\ldots,n-4$ in (\ref{main_eq}), we get
\begin{align}\label{system}
\nonumber
&s_0\f(4)=\f(0)+(-s_3+x_3y_0+x_2y_1)\f(1)+(-s_2+x_2y_0)\f(2)-s_1\f(3),\\
\nonumber
&s_0\f(5)=(1-s_4+x_4y_0+x_3y_1)\f(1)+(-s_3+x_3y_0)\f(2)-s_2\f(3)-s_1\f(4),\\ 
&\vdots\\
&s_0\f(n)=(x_{n-1}y_0+x_{n-2}y_1)\f(1)+x_{n-2}y_0\f(2)-\f(n-4)-\sum_{k=1}^{n-1}\f(k)s_{n-k}.\nonumber
\end{align}
The obtained list of equations in (\ref{system}), allows to express $\f(n)$ via $\f(0),\,\ldots,\,\f(3)$ for $n\in\mathbb{N}_0$. Also, it is seen that to get $\f(4)$ and later $\f(5),\,\f(6),\, \ldots$ it is needed to know the first ones $\f(0),\,\ldots,\,\f(3)$. Therefore, the main problem we deal with in this paper is the finding of needed initial values for recurrence relation in (\ref{main_eq}). As equations in (\ref{system}) are dependent on PDF of convolution of r.vs. $X$ and $Y$, the expressions of needed initial values are also heavily dependent on the structure of $X+Y$. With that said, currently we do not develop any wider generalizations for the model in (\ref{model}) than $\kappa=2$ and $Z_{2i-1}{\buildrel d \over =}X$, $Z_{2i}{\buildrel d \over =}Y,\,i\in\mathbb{N}$. The natural parameter $\kappa \geqslant 3$ and general number of independent but different discrete non-negative r.vs. generating the r.w. $\sum_{i=1}^{t}Z_i$ interchangeably, would significantly increase the level of abstractions. Due to that, the present paper and \cite{DS} are just a starting points to confirm the correctness and provide a direction if any generalizations for the model in (\ref{model}) are developed in the future. The research done in \cite{GKS} may also be deemed as initial step to generalization of model (\ref{model}) as three-seasonal model with $\kappa=1$ was studied there. 

It is curious that the number of initial values, needed for (\ref{main_eq}), might be reduced by one, finding the relation of $\f(0),\,\dots,\f(3)$. Before demonstrating that, we introduce the net profit condition. It is said that the {\it net profit condition} holds for the model (\ref{model}) with $\kappa=2$ if $\mathbb{E}S<4$. This condition is crucial trying to avoid guaranteed ruin (survival with probability 0) as time grows ultimately, see Theorem \ref{no net} in Section \ref{sec:Statements}. An intuitive understanding of the net profit condition is simple. The expectation of $W(t)$ in (\ref{model}) for even or odd $t\in\mathbb{N}$ is:
\begin{align*}
&\mathbb{E}W(2t)=u+t(2\kappa-\mathbb{E}S),\\
&\mathbb{E}W(2t-1)=u+\kappa+\mathbb{E}X+t(2\kappa-\mathbb{E}S),
\end{align*}
where the sign of $2\kappa-\mathbb{E}S$ influences the sign of $\mathbb{E}W(t)$ and possibility that $W(t)>0$ for some natural $t$'s.

We now turn back to the relation of $\f(0),\,\dots,\,\f(3)$. By summing up the both sides of (\ref{main_eq}) by $u$ from $0$ to some sufficiently large natural $v$, we get
\begin{align}\label{sum}
&\sum_{u=0}^{v}\f(u)=\sum_{u=0}^{v}\sum_{k=1}^{u+4}\varphi(k)s_{u+4-k}
-\left((X(v+3)-X(2))y_0+(X(v+2)-X(1))y_1\right)\varphi(1)\nonumber\\
&-((X(v+2)-X(1))y_0)\varphi(2)
,
\end{align}
where
\begin{align}\label{change_sum_order}
&\sum_{u=0}^{v}\sum_{k=1}^{u+4}\varphi(k)s_{u+4-k}\nonumber
=\sum_{k=1}^{3}\f(k)\sum_{u=0}^{v}s_{u+4-k}
+\sum_{k=4}^{v+4}\f(k)\sum_{u=k-4}^{v}s_{u+4-k}\\
&=\sum_{k=1}^{3}\f(k)\left(S(v+4-k)-S(3-k)\right)
+\sum_{k=4}^{v+4}\f(k)S(v+4-k).
\end{align}
Inserting (\ref{change_sum_order}) into (\ref{sum}) and rearranging we obtain
\begin{align}\label{eq:before_l}
\nonumber
&\sum_{k=0}^{v+4}\f(k)\overline{S}(v+4-k)-\sum_{u=v+1}^{v+4}\f(u)
=\sum_{k=1}^{3}\f(k)\left(S(v+4-k)-S(3-k)\right)
-\sum_{k=0}^{3}\f(k)S(v+4-k)\\
&
-\left((X(v+3)-X(2))y_0+(X(v+2)-X(1))y_1\right)\varphi(1)
-((X(v+2)-X(1))y_0)\varphi(2).
\end{align}
If $v\to\infty$ and $\mathbb{E}S<4$, from the last equation, by Lemma \ref{lemma}, we get
\begin{align}\label{eq:one less}
\f(0)+\left(\overline{X}(2)y_0+\overline{X}(1)y_1\right)\f(1)+\overline{X}(1)y_0\f(2)+\sum_{k=1}^{3}\f(k)S(3-k)=4-\E S.    
\end{align}

Adding equality (\ref{eq:one less}) to the list of equations in (\ref{system}), we are able to express $\f(n)$ via $\f(0),\, \f(1)$ and $\f(2)$ for all $n\in\mathbb{N}_0$. Namely that is the main idea for finding a needed initial values for the recurrence relation in (\ref{main_eq}).

The rest of the paper is structured as follows. In Section \ref{sec:Statements}, Theorem \ref{finite} serves purpose for the expression of finite time survival probability $\f(u,T)$ while Theorems \ref{T_3x3} -- \ref{T_0x0} deal with expressions of the ultimate time survival probability under the net profit condition. Expressions of $\f$ under $\mathbb{E}S<4$ in Theorems \ref{T_3x3} -- \ref{T_0x0} are dependent on the lowest value of the distribution of $S=X+Y$. The last Theorem \ref{no net} demonstrates that survival is impossible in all except few trivial cases if the net profit condition is violated $\mathbb{E}S\geqslant4$. 

\section{Statements}\label{sec:Statements}
In this section we formulate all of the statements which are necessary to express finite and ultimate time survival probabilities $\f(u,T)$ and $\f(u)$ of the model (\ref{model}) with $\kappa=2$. In fact, the distribution function $\f(u,T)$ can be calculated by following the ideas in \cite[Theorems 1--4]{BBS}, however Theorem \ref{finite} below is specially adopted for the bi-seasonal discrete time risk model with income rate two. Our reason to have it, is an interest to the broader view as $T$ grows and $\f(u,T)$ approximates $\f(u)$, see Section \ref{sec:calculations}. All of the statements formulated in this section are proved in the later Section \ref{sec:proofs}.

\begin{thm}\label{finite}
For the finite time survival probability $\f(u,T)$ of the bi-seasonal discrete time risk model with income rate two, holds: 
\begin{align*}
&\f(u,1)=X(u+1),\,
\f(u,2)=\sum_{k=0}^{u+1}x_k Y(u+3-k),\\
&\f(u,T)=\sum_{k=0}^{u+3}\f(u+4-k,T-2)s_k-x_{u+2}y_0\f(2,T-2)
-(x_{u+2}y_1+x_{u+3}y_0)\f(1,T-2),\,T\geqslant 3.
\end{align*}
\end{thm}

We now turn to the ultimate time. As mentioned, expressions of $\f$ are heavily dependent on the lowest value of convolution $S=X+Y$. For $s_0>0$ let us define four recurrent sequences $\alpha_n,\,\beta_n,\,\gamma_n,\,\delta_n$. For $n=0,1,2,3$
$$
\begin{tabular}{|c||c|c|c||c|}
\hline
${n}$&$\alpha_n$&$\beta_n$&$\gamma_n$&$\delta_n$\\
\hline
\hline
${0}$&$1$&$0$&$0$&$0$\\
\hline
${1}$&$0$&$1$&$0$&$0$\\
\hline
$2$&$0$&$0$&$1$&$0$ \\
\hline
$3$&$-{1}/{s_0}$ &$-\left((S(2)+\overline{X}(2)y_0+\overline{X}(1)y_1\right)/{s_0}$ & $-\left(S(1)+\overline{X}(1)y_0\right)/{s_0}$&${1}/{s_0}$ \\
\hline
\end{tabular}
$$
and for $n=4,5,\,\ldots$
\begin{align*}
&\alpha_n=\frac{1}{s_0}\left(\alpha_{n-4}-\sum_{k=1}^{n-1}s_{n-k}\alpha_k\right),\,
\beta_n=\frac{1}{s_0}\left(\beta_{n-4}-\sum_{k=1}^{n-1}s_{n-k}\beta_k+x_{n-1}y_0+x_{n-2}y_1\right),\\
&\gamma_n=\frac{1}{s_0}\left(\gamma_{n-4}-\sum_{k=1}^{n-1}s_{n-k}\gamma_k+x_{n-2}y_0\right),\,
\delta_n=\frac{1}{s_0}\left(\delta_{n-4}-\sum_{k=1}^{n-1}s_{n-k}\delta_k\right).
\end{align*}
\begin{thm}\label{T_3x3}
If $s_0>0$ and $\mathbb{E}S<4$, for the ultimate time survival probability of the bi-seasonal discrete time risk model with income rate two, holds:
\begin{align}\label{3x3}
\left( \begin{array}{ccc}
\alpha_{n+1}-\alpha_{n} & \beta_{n+1}-\beta_{n} & \gamma_{n+1}-\gamma_{n} \\
\alpha_{n+2}-\alpha_{n} & \beta_{n+2}-\beta_{n} & \gamma_{n+2}-\gamma_{n} \\
\alpha_{n+3}-\alpha_{n} & \beta_{n+3}-\beta_{n} & \gamma_{n+3}-\gamma_{n} \\
\end{array} \right)
\times
\left( \begin{array}{c}
\varphi(0) \\
\varphi(1) \\
\varphi(2)
\end{array} \right)
+&
\left( \begin{array}{c}
\delta_{n+1}-\delta_{n} \\
\delta_{n+2}-\delta_{n} \\
\delta_{n+3}-\delta_{n}
\end{array} \right)
\times
(4-\mathbb{E}S)\nonumber \\
&=
\left( \begin{array}{c}
\varphi(n+1)-\varphi(n) \\
\varphi(n+2)-\varphi(n)\\
\varphi(n+3)-\varphi(n)
\end{array} \right),\,n\in\mathbb{N}_0,
\end{align}
\end{thm}

\begin{align*}
&\f(3)=\frac{-\f(0)-(S(2)+\overline{X}(2)y_0+\overline{X}(1)y_1)\f(1)-(S(1)+\overline{X}(1)y_0)\f(2)+4-\E S}{s_0},\\
&\f(u)=\frac{1}{s_0}\left(\f(u-4)+(x_{u-1}y_0+x_{u-2}y_1)\f(1)+x_{u-2}y_0\f(2)
-\sum_{k=1}^{u-1}s_{u-k}\f(k)
\right),\,u=4,\,5,\,\ldots
\end{align*}

\begin{remark}\label{remark1}
It is evident that $\f(n)\leqslant\f(n+1)\leqslant1$ for all $n\in\mathbb{N}_0$. Also, Lemma \ref{lemma} implies $\f(n)\approx1$ if $n$ is sufficiently large. Therefore, in practical applications the right hand side of (\ref{3x3}) is assumed $(0,0,0)^T$ and three needed initial values $\f(0),\,\f(1),\,\f(2)$ are solved out from (\ref{3x3}).
\end{remark}

\begin{remark}\label{remark2}
We can not prove the system matrix in (\ref{3x3}) being non-singular for all $n\in\mathbb{N}_0$. On the other hand, we never find such matrix being singular with any chosen underlying distributions. Attempts to prove and numerical calculations raise the following conjecture.
\end{remark}

\begin{con}\label{con1}
Let $D_n$ denote the principal determinant of the system matrix in (\ref{3x3}). Then, $1\leqslant D_{2n}\leqslant D_{2n+2}$ and $-1\geqslant D_{2n+1}\geqslant D_{2n+3}$  for all $n\in\mathbb{N}_0$.
\end{con}

We now assume $s_0=0$ and $s_1>0$. For $s_1>0$ let us define three recurrent sequences $\overline{\alpha}_n,\,\overline{\beta}_n,\,\overline{\delta}_n$. For $n=0,1,2$
$$
\begin{tabular}{|c||c|c||c|}
\hline
${n}$&$\overline{\alpha}_n$&$\overline{\beta_n}$&$\overline{\delta}_n$\\
\hline
\hline
${0}$&$1$&$0$&$0$\\
\hline
${1}$&$0$&$1$&$0$\\
\hline
$2$&$-{1}/(s_1+\overline{X}(1)y_0)$ &$-(S(2)+\overline{X}(2)y_0+\overline{X}(1)y_1)/({s_1+\overline{X}(1)y_0)}$ &${1}/(s_1+\overline{X}(1)y_0)$ \\
\hline
\end{tabular}
$$
and for $n=3,4,\,\ldots$
\begin{align*}
&\overline{\alpha}_n=\frac{1}{s_1}\left(\alpha_{n-3}-\sum_{k=1}^{n-1}s_{n+1-k}\overline{\alpha}_k+x_{n-1}y_0\overline{\alpha}_2 \right),\\
&\overline{\beta}_n=\frac{1}{s_1}\left(\overline{\beta}_{n-3}-\sum_{k=1}^{n-1}s_{n+1-k}\overline{\beta}_k+x_{n}y_0
+x_{n-1}y_1+x_{n-1}y_0\overline{\beta}_2
\right),\\
&\overline{\delta}_n=\frac{1}{s_0}\left(\overline{\delta}_{n-3}-\sum_{k=1}^{n-1}s_{n+1-k}\overline{\delta}_k
+x_{n-1}y_0\overline{\delta}_2
\right).
\end{align*}
\begin{thm}\label{T_2x2}
If $s_0=0$, $s_1>0$ and $\mathbb{E}S<4$, for the ultimate time survival probability of the bi-seasonal discrete time risk model with income rate two, holds:
\begin{align}\label{2x2}
\left( \begin{array}{cc}
\overline{\alpha}_{n+1}-\overline{\alpha}_{n} & \overline{\beta}_{n+1}-\overline{\beta}_{n} \\
\overline{\alpha}_{n+2}-\overline{\alpha}_{n} & \overline{\beta}_{n+2}-\overline{\beta}_{n} \\
\end{array} \right)
\times
\left( \begin{array}{c}
\varphi(0) \\
\varphi(1)
\end{array} \right)
+
\left( \begin{array}{c}
\overline{\delta}_{n+1}-\overline{\delta}_{n} \\
\overline{\delta}_{n+2}-\overline{\delta}_{n}
\end{array} \right)
\times
(4-\mathbb{E}S)
=
\left( \begin{array}{c}
\varphi(n+1)-\varphi(n) \\
\varphi(n+2)-\varphi(n)\\
\end{array} \right),\,n\in\mathbb{N}_0,
\end{align}
\end{thm}
\begin{align*}
&\f(2)=\frac{-\f(0)-(S(2)+\overline{X}(2)y_0+\overline{X}(1)y_1)\f(1)+4-\E S}{s_1+\overline{X}(1)y_0},\\
&\f(u)=\frac{1}{s_1}\left(\f(u-3)+(x_{u}y_0+x_{u-1}y_1)\f(1)
+x_{u-1}y_0\f(2)
-\sum_{k=1}^{u-1}s_{u+1-k}\f(k)
\right),\,u=3,4,\,\ldots
\end{align*}

\begin{remark}\label{remark3}
We note that conditions $s_0=0$ and $s_1>0$ imply that denominator of $\f(2)$ in Theorem \ref{T_2x2} and of expressions $\overline{\alpha}_2,\,\overline{\beta}_2,\, \overline{\gamma}_2$ above, is $s_1+\overline{X}(1)y_0=y_0+x_0y_1>0$. 
\end{remark}

\begin{remark}\label{remark4}
Practical applications of Theorem \ref{T_2x2} are described the same way as for Theorem \ref{T_3x3} in Remark \ref{remark1}. The non-vanishing of system matrix in (\ref{2x2}) is unknown for all $n\in\mathbb{N}_0$. Arguing the same as in Remark \ref{remark2}, the following conjecture is raised.
\end{remark}

\begin{con}\label{con2}
Let $\overline{D}_n$ denote the principal determinant of the system matrix in (\ref{2x2}). Then, $1\leqslant \overline{D}_n\leqslant \overline{D}_{n+1}$ for all $n\in\mathbb{N}_0$.
\end{con}

We next turn to the case $s_0=s_1=0,\,s_2>0$. This, in turn, has three underlying scenarios which impact the expression of survival probability $\varphi$:

\begin{itemize}
  \item[s.1]  $x_0=0,\,y_0=0,\,x_1>0,\,y_1>0$,
  \item[s.2]  $x_0>0,\,y_0=0,\,y_1=0,\,y_2>0$,
  \item[s.3]  $x_0=0,\,y_0>0,\,x_1=0,\,x_2>0$.
\end{itemize}

Scenarios s.1 and s.2 imply that $s_2-x_2y_0=x_1y_1+x_0y_2>0$, while s.3 implies $s_2-x_2y_0=0$. Under s.1 or s.2 let us define two recurrent sequences:
\begin{align*}
&\hat{\alpha}_0=1,\,
\hat{\alpha}_1=-\frac{1+\frac{\overline{X}(1)y_0}{s_2-x_2y_0}}{\overline{X}(2)y_0+\overline{X}(1)y_1+s_2-
\frac{\overline{X}(1)y_0(s_3-x_3y_0-x_2y_1)}{s_2-x_2y_0}}
=-\frac{1}{\overline{X}(1)y_1+s_2}
,\\
&\hat{\alpha}_n=\frac{1}{s_2}\left(\hat{\alpha}_{n-2}-\sum_{k=1}^{n-1}\hat{\alpha}_k s_{n+2-k}+x_n y_1\hat{\alpha}_1\right) ,\,n=2,3,\,\ldots\\
&\hat{\delta}_0=0,
\,\hat{\delta}_1=\frac{1}{\overline{X}(2)y_0+\overline{X}(1)y_1+s_2-
\frac{\overline{X}(1)y_0(s_3-x_3y_0-x_2y_1)}{s_2-x_2y_0}}
=\frac{1}{\overline{X}(1)y_1+s_2}
,\\
&\hat{\delta}_n=\frac{1}{s_2}\left(\hat{\delta}_{n-2}-\sum_{k=1}^{n-1}\hat{\delta}_k s_{n+2-k}+x_n y_1\hat{\delta}_1\right) ,\,n=2,3,\,\ldots
\end{align*}

\begin{remark}\label{rem2}
We note that s.1 and s.2 accordingly imply $\hat{\alpha}_1=-1/y_1=-\hat{\delta}_1$ and $\hat{\alpha}_1=-1/(x_0y_2)=-\hat{\delta}_1$.
\end{remark}

Under s.3 let us also define two recurrent sequences:
\begin{align*}
&\tilde{\alpha}_1=1,\,
\tilde{\alpha}_2=-\frac{y_0+y_1}{y_0},\,
\tilde{\alpha}_n=\frac{1}{s_2}\left(\tilde{\alpha}_{n-2}-\sum_{k=1}^{n-1}\tilde{\alpha}_k s_{n+2-k}+(x_{n+1}-x_n)y_0\right) ,\,n=3,4,\,\ldots\\
&\tilde{\delta}_1=0,
\,\tilde{\delta}_2=\frac{1}{y_0},\,
\tilde{\delta}_n=\frac{1}{s_2}\left(\tilde{\delta}_{n-2}-\sum_{k=1}^{n-1}\tilde{\delta}_k s_{n+2-k}+x_n\right) ,\,n=3,4,\,\ldots
\end{align*}

\begin{thm}\label{T_1x1}
If $s_0=s_1=0$, $s_2>0$ and $\mathbb{E}S<4$, for the ultimate time survival probability of the bi-seasonal discrete time risk model with income rate two, holds:

Under s.1 or s.2:
\begin{align}\label{1x1}
&(\hat{\alpha}_{n+1}-\hat{\alpha}_{n})\f(0)+(\hat{\delta}_{n+1}-\hat{\delta}_{n})(4-\E S)=\f(n+1)-\f(n),\,n\in\mathbb{N}_0\\ 
\nonumber
&\f(1)=\hat{\alpha}_1\f(0)+\hat{\delta}_1(4-\E S).
\end{align}
Under s.3:
\begin{align}\label{1x1_v1}
\f(0)=0,\,
(\tilde{\alpha}_{n+1}-\tilde{\alpha}_{n})\f(1)+(\tilde{\delta}_{n+1}-\tilde{\delta}_{n})(4-\E S)=\f(n+1)-\f(n),\,n\in\mathbb{N}.
\end{align}
The remaining values of the survival probability are calculated by
\begin{align*}
\varphi(u) = \frac{1}{s_2}\left(\varphi(u-2)- \sum\limits_{k=1}^{n-1} \varphi(k)s_{u+2-k} + (x_{u+1}y_0+x_{u}y_1)\varphi(1)+x_{u}y_0\varphi(2) \right),\,u=2,\,3,\,\ldots
\end{align*}
In addition, $\hat{\alpha}_{n+1}-\hat{\alpha}_{n}\neq0$ for all $n\in\mathbb{N}_0$ and $\tilde{\alpha}_{n+1}-\tilde{\alpha}_{n}\neq0$ for all $n\in\mathbb{N}$.
\end{thm}

\begin{remark}
As commented in Remark \ref{remark1}, the practical application of Theorem \ref{T_1x1} is based on $\f(n+1)-\f(n)\approx0$ when $n$ is sufficiently large. 
\end{remark}

We now turn to the last case of $\varphi$ dependencies on r.v. $S$, which is $s_0=s_1=s_2=0$ and $s_3>0$. This, in turn, has four underlying scenarios, which impact an expressions of $\varphi$:\\
v.1 $x_0=0,\,y_0=0,\,x_1=0,\,y_1>0,\,x_2>0$,\\
v.2 $x_0=0,\,y_0=0,\,y_1=0,\,x_1>0,\,y_2>0$,\\
v.3 $x_0=0,\,y_0>0,\,x_1=0,\,x_2=0,\,x_3>0$,\\
v.4 $x_0>0,\,y_0=0,\,y_1=0,\,y_2=0,\,y_3>0$.

Formulas of $\f$ under scenarios form v.1 to v.4 dictate a need to evaluate $s_3-x_3y_0$ and $s_3-x_3y_0-x_2y_1$. We do that in the following table:

\begin{center}
\begin{tabular}{|c||c|c|} 
 \hline
 Scenario & $s_3-x_3y_0$ & $s_3-x_3y_0-x_2y_1$ \\ 
  \hline
  \hline
 v.1 & $>0$ & $=0$ \\ 
  \hline
 v.2 & $>0$ & $>0$ \\
 \hline
 v.3 & $=0$ & $=0$ \\
 \hline
 v.4 & $>0$ & $>0$ \\
 \hline
\end{tabular}
\end{center}

\begin{thm}\label{T_0x0}
If $s_0=s_1=s_2=0$, $s_3>0$ and $\mathbb{E}S<4$, for the ultimate time survival probability of the bi-seasonal discrete time risk model with income rate two, holds:

\begin{align*}
&\text{Under v.2 or v.4: }
\f(0)=\frac{4-\E S}{1+\frac{\overline{X}(1)y_1}{x_1y_2+x_0y_3}},\,
\f(1)=\frac{\f(0)}{x_1y_0+x_0y_3},\\
&\text{Under v.1: } \f(0)=0,\,\f(1)=\frac{4-\mathbb{E}S}{y_1},\\
&\text{Under v.3: } \f(0)=\f(1)=0,\, \f(2)=\frac{4-\mathbb{E}S}{y_0}.
\end{align*}
The remaining values of the survival probability are calculated by
\begin{align*}
&\f(u)=\frac{1}{s_3}\left(\f(u-1)+(x_{u+2}y_0+x_{u+1}y_1)\f(1)+x_{u+1}y_0\f(2)
-\sum_{k=1}^{u-1}s_{u+3-k}\f(k)
\right),\,u=2,\,3,\ldots
\end{align*}
\end{thm}
It is easy to see that $s_0=\ldots=s_3=0$ leads to the unsatisfied net profit condition $\mathbb{E}S\geqslant4$. If that happens, the following statement is true. 
\begin{thm}\label{no net}
With the unsatisfied net profit condition 
for the ultimate time survival probability of the bi-seasonal discrete time risk model with income rate two,
holds:
\begin{itemize}
\item $\varphi(u)=0$ for all $u\in\mathbb{N}_0$ if $\mathbb{E}S>4$,
\item $\varphi(u)=0$ for all $u\in\mathbb{N}_0$ if $\mathbb{E}S=4$ and $s_4<1$,
\item  If $\mathbb{E}S=4$ and $s_4=1$, then the following sub-cases arise:
\begin{itemize}
\item[$\diamond$] $\varphi(0)=\varphi(1)=\varphi(2)=0$ and $\varphi(u)=1$ for $u\geqslant3$ if $x_4=y_0=1$,
\item[$\diamond$] $\varphi(0)=\varphi(1)=0$ and $\varphi(u)=1$ for $u\geqslant2$ if $x_3=y_1=1$,
\item[$\diamond$] $\varphi(0)=0$ and $\varphi(u)=1$ for $u\geqslant1$ if $x_2=y_2=1$ or $x_1=y_3=1$ or $x_0=y_4=1$.
\end{itemize}
\end{itemize}
\end{thm}

It is worth mentioning that a similar model to (\ref{model}), with $\kappa=2$ and $X\dist{Y}$, was studied in \cite[Section 2]{GS} where similar recurrent matrices $2\times2$ as in (\ref{2x2}) were obtained. These matrices recently have been studied in \cite{GJ},  where some results on its non-singularity were obtained. Moreover, it was shown in \cite{GJ} that a required initial values of survival probability for a homogeneous discrete time risk model ($X\dist{Y}$) with premium rate two, have expressions via certain roots of probability generating function. It is very likely that the same ideas as in \cite{GJ} are applicable to Conjectures \ref{con1} and \ref{con2}, and also for finding the exact initial values of $\varphi$, which are approximately expressed in Theorems \ref{T_3x3}--\ref{T_1x1}. On the other hand, approximate expressions of $\varphi$ as given in Theorems \ref{T_3x3}--\ref{T_1x1} do not require nor existence of probability generating functions, nor any knowledge on the roots of a certain type of power series. 

\section{Proofs}\label{sec:proofs}

In this section we prove all of the statements formulated in the previous Section \ref{sec:Statements}. We start with an auxiliary lemma on survival probability $\f$. Let us denote $\f(\infty):=\lim_{u\to\infty}\f(u)$.

\begin{lem}\label{lemma}
For the ultimate time survival probability of the bi-seasonal discrete time risk model with income rate two, the following relations hold:
\begin{eqnarray}
&&\varphi(\infty)=1, \text{ if } \mathbb{E}S<4, \label{to1} \\
&&\lim\limits_{v\to \infty}\sum_{k=0}^{v+4}\varphi(k) \overline{S}(v+4-k)=\varphi(\infty)\cdot\mathbb{E}S. \label{expectation}
\end{eqnarray}
\end{lem}

\begin{proof}
The first equality (\ref{to1}) is implied by replicating the proof line by line of the bi-seasonal model with income rate one in \cite[p. 935--936]{DS}. 

The remaining equality (\ref{expectation}) follows by the fact that the lower and upper bounds of
\begin{align*}
\lim\limits_{v\to \infty}\sum_{k=0}^{v+4}\varphi(k) \overline{S}(v+4-k)
\end{align*}
are the same, see \cite[p. 4]{GS} or \cite[p. 13]{GS_Math}.
\end{proof}

\begin{proof}[Proof of Theorem \ref{finite}]
Definition of the finite-time ruin probability 
\begin{align*}
\varphi(u,T)= \T \left( \bigcap\limits_{t=1}^{T} \{ W(t) >0 \} \right)
\end{align*}
implies
\begin{align*}
\varphi(u,1)=\T \left( u+2-Z_1 > 0 \right) = X(u+1).
\end{align*}
In the same manner for $T=2$
\begin{equation*}
\begin{split}
& \varphi(u,2) = \T ( \{ W(1)>0 \} \cap \{ W(2)\}) = \T ( \{Z_1 < u+2 \} \cap \{Z_1 + Z_2 < u+4 \}) \\
&= \T(\{X \leqslant u+1 \} \cap \{ X+Y \leqslant u+3 \})= \sum\limits_{k=0}^{u+1} \T( X = k) \T (Y \leqslant u+3-k) = \sum\limits_{k=0}^{u+1} x_k Y(u+3-k).
\end{split}
\end{equation*}
For $T\geqslant 3$, by the similar arguments as obtaining (\ref{main_eq}), we get
\begin{equation*}
\begin{split}
& \varphi(u,T) = \T \left(\bigcap\limits_{t=1}^{T} \Bigg\{ u+2t - \sum\limits_{i=1}^t Z_i >0 \Bigg\} \right) = \T \left(\bigcap\limits_{t=2}^{T} \Bigg\{ u+2t - \sum\limits_{i=1}^t Z_i >0 \Bigg\} \cap \{ Z_1 <u+2 \} \right) \\
&= \T \left(\bigcap\limits_{t=2}^{T} \Bigg\{ u+2t - \sum\limits_{i=1}^t Z_i >0 \Bigg\} \right) - \T \left(\bigcap\limits_{t=2}^{T} \Bigg\{ u+2t - \sum\limits_{i=1}^t Z_i >0 \Bigg\} \cap \{ Z_1 \geqslant u+2 \} \right) \\
& = \T \left(\bigcap\limits_{t=3}^{T} \Bigg\{ u+2t -Z_1 -Z_2- \sum\limits_{i=3}^t Z_i >0 \Bigg\} \cap \{ Z_1+Z_2 <u+4 \} \right) \\
&- \T \left(\bigcap\limits_{t=3}^{T} \Bigg\{ u+2t -Z_1 -Z_2- \sum\limits_{i=3}^t Z_i >0 \Bigg\} \cap\{Z_1 \geqslant u+2\} \cap \{ Z_1+Z_2 <u+4 \} \right) \\
&= \sum\limits_{k=0}^{u+3} \T (Z_1 +Z_2 =k) \T \left(\bigcap\limits_{t-2=1}^{T-2} \Bigg\{ u+2(t-2)+4 -k - \sum\limits_{i=1}^{t-2} Z_i >0 \Bigg\} \cap \{ Z_1+Z_2 =k \} \right) \\
&- (x_{u+2}y_1+x_{u+3}y_0) \T \left(\bigcap\limits_{t=1}^{T-2} \Bigg\{ 1+2t - \sum\limits_{i=1}^t Z_i >0 \Bigg\} \right) -x_{u+2}y_0 \T \left(\bigcap\limits_{t=1}^{T-2} \Bigg\{ 2+2t - \sum\limits_{i=1}^t Z_i >0 \Bigg\} \right)\\
&= \sum\limits_{k=0}^{u+3} \varphi(u+4-k,T-2)s_k - (x_{u+2}y_1+x_{u+3}y_0) \varphi(1,T-2) - x_{u+2}y_0 \varphi(2,T-2).
\end{split}
\end{equation*}
\end{proof}

\begin{proof}[Proof of Theorem \ref{T_3x3}]
Let $\alpha_n$, $\beta_n$, $\gamma_n$ and $\delta_n$ be the recurrent sequences defined prior to Theorem \ref{T_3x3}. We aim to show that for all $n\in\mathbb{N}_0$ 
\begin{equation}\label{eq}
\varphi(n)= \alpha_n \varphi(0) + \beta_n \varphi(1) + \gamma_n \varphi(2) + \delta_n (4- \E S).
\end{equation}
If $n=0,\,1$ or $n=2$, the statement is evident. If $n=3$, the relation results from (\ref{eq:one less})
\begin{equation*}
\begin{split}
&\varphi(3) = -\frac{1}{s_0} \varphi(0) -\frac{\overline{X}(2)y_0+\overline{X}(1)y_1+S(2)}{s_0} \varphi(1) -\frac{\overline{X}(1)y_0 +S(1)}{s_0}\varphi(2) + \frac{1}{s_0}(4- \E S)\\
&=\alpha_3 \varphi(0) + \beta_3 \varphi(1) + \gamma_3 \varphi(2) + \delta_3 (4- \E S).
\end{split}
\end{equation*}
For $T \geqslant 4$ we use induction. Then, by (\ref{main_eq}) and induction hypothesis

\begin{equation*}
\begin{split}
&\varphi(n) = \frac{1}{s_0}\left(\varphi(n-4) + (x_{n-1}y_0+x_{n-2}y_1)\varphi(1)+x_{n-2}y_0\varphi(2) - \sum\limits_{k=1}^{n-1} \varphi(k)s_{n-k}\right)\\
& = \frac{1}{s_0}(\alpha_{n-4}\varphi(0) +\beta_{n-4}\varphi(1)+\gamma_{n-4}\varphi(2)+\delta_{n-4}(4-\E S) \\
&+ (x_{n-1}y_0+x_{n-2}y_1)(\alpha_{1}\varphi(0) +\beta_{1}\varphi(1)+\gamma_{1}\varphi(2)+\delta_{1}(4-\E S)) \\
&+x_{n-2}y_0(\alpha_{2}\varphi(0) +\beta_{2}\varphi(1)+\gamma_{2}\varphi(2)+\delta_{2}(4-\E S)) \\
&-\sum\limits_{k=1}^{n-1} s_{n-k}(\alpha_{k}\varphi(0) +\beta_{k}\varphi(1)+\gamma_{k}\varphi(2)+\delta_{k}(4-\E S)) \\
&= \varphi(0) \frac{1}{s_0}\left( \alpha_{n-4} - \sum\limits_{k=1}^{n-1} s_{n-k}\alpha_k \right) + \varphi(1) \frac{1}{s_0}\left( \beta_{n-4} -\sum\limits_{k=1}^{n-1} s_{n-k}\beta_k+ x_{n-1}y_0+x_{n-2}y_1 \right) \\
&+ \varphi(2)\frac{1}{s_0}\left( \gamma_{n-4} -\sum\limits_{k=1}^{n-1} s_{n-k}\gamma_k+ x_{n-2}y_1 \right) + (4-\E S)\frac{1}{s_0}\left( \delta_{n-4} -\sum\limits_{k=1}^{n-1} s_{n-k}\delta_k\right) \\
&= \alpha_n \varphi(0) + \beta_n \varphi(1) + \gamma_n \varphi(2) + \delta_n (4- \E S).
\end{split}
\end{equation*}
Consequently, equation (\ref{eq}) holds for all $n \in \mathbb{N}_0$. By determining the differences $\varphi(n+1) - \varphi(n),\\ \varphi(n+2) - \varphi(n)$ and $\varphi(n+3) - \varphi(n)$ we obtain the system (\ref{3x3}) where the remaining equalities in Theorem \ref{T_3x3} are implied by (\ref{eq:one less}) and (\ref{main_eq}).

\end{proof}

\begin{proof}[Proof of Theorem \ref{T_2x2}]
Let $\overline{\alpha}_n, \overline{\beta}_n$ and $\overline{\delta}_n$ be the recurrent sequences defined prior to Theorem \ref{T_2x2}. Then, arguing the same as in proof of Theorem \ref{T_3x3}, we show that 
\begin{equation}\label{eq2}
\varphi(n)= \overline{\alpha}_n \varphi(0) + \overline{\beta}_n \varphi(1) + \overline{\delta}_n (4- \mathbb{E} S),\,n\in\mathbb{N}_0.
\end{equation}
If $n=0$, or $n=1$, the equation is obvious. For $n=2$, it follows by (\ref{eq:one less}) 
\begin{equation*}
\begin{split}
& \varphi(2) = -\frac{1}{s_1 + \overline{X}(1)y_0} \varphi(0) -\frac{\overline{X}(2)y_0+\overline{X}(1)y_1+S(2)}{s_1 + \overline{X}(1)y_0} \varphi(1) + \frac{1}{s_1 + \overline{X}(1)y_0}(4- \E S)\\
&=\overline{\alpha}_2 \varphi(0) + \overline{\beta}_2 \varphi(1) + \overline{\delta}_2 (4- \E S)
\end{split}
\end{equation*}
and, for any $n\in\mathbb{N}_0$, induction does its work as in the previous proof and the rest is implied by (\ref{eq:one less}) and (\ref{main_eq}).
\end{proof}

\begin{proof}[Proof of Theorem \ref{T_1x1}]
Let us consider the cases s.1 and s.2. Recall that $\hat{\alpha}_n$ and $\hat{\delta}_n$ are recurrent sequences defined prior to Theorem \ref{T_1x1}. Then, the following equality is true
\begin{equation}\label{eq4t}
\varphi(n)= \hat{\alpha}_n \varphi(0) + \hat{\delta}_n (4- \E S),\,
n\in\mathbb{N}_0.
\end{equation}
Indeed, if $n=0$ the equation is evident, if $n=1$, from (\ref{eq:one less}) we get
\begin{equation}\label{l1}
\varphi(2) = \frac{1}{\overline{X}(1)y_0}\left( -\varphi(0) -(\overline{X}(2)y_0+\overline{X}(1)y_1+s_2)\varphi(1) + 4- \E S \right).
\end{equation}
By setting $n=0$ into (\ref{main_eq}) we get
\begin{equation}\label{l2}
\varphi(2) = \frac{1}{s_2-x_2y_0}\left( \varphi(0) - (s_3 -x_3y_0 -x_2y_1)\varphi(1) \right).
\end{equation}
Equating (\ref{l1}) to (\ref{l2}) and rearranging we obtain
\begin{equation*}
\begin{split}
&\varphi(1)= -\varphi(0)\frac{1+ \frac{\overline{X}(1)y_0}{s_2-x_2y_0}}{\overline{X}(2)y_0+\overline{X}(1)y_1+s_2 -\frac{\overline{X}(1)y_0(s_3 -x_3y_0 -x_2y_1)}{s_2-x_2y_0}}\\
&+\frac{4-\E S}{\overline{X}(2)y_0+\overline{X}(1)y_1+s_2 -\frac{\overline{X}(1)y_0(s_3 -x_3y_0 -x_2y_1)}{s_2-x_2y_0}} = \hat{\alpha}_1 \varphi(0) + \hat{\delta}_1 (4-\E S).
\end{split}
\end{equation*}
Observing that $y_0=0$ under scenarios s.1 and s.2, we confirm the equation (\ref{eq4t}) for $n=1$. 

For $n \geqslant 2$, eq. (\ref{eq4t}) follows by mathematical induction the same way as proving Theorems \ref{T_3x3} and \ref{T_2x2}.
By setting the difference $\varphi(n+1) - \varphi(n)$ we obtain the equation defined in Theorem \ref{T_1x1}
$$(\hat{\alpha}_{n+1}-\hat{\alpha}_n)\varphi(0)+(\hat{\delta}_{n+1}-\hat{\delta}_n)(4-\mathbb{E}S) = \varphi(n+1)-\varphi(n).$$
Considering the last case s.3 we note that $\varphi(0) = 0$. This is due to the first possible claim (r.v. $X$) at $t=1$ is greater or equal to $2$. For the following survival probabilities $\f(u),\,u\geqslant1$, let $\tilde{\alpha}_n$ and $\tilde{\delta}_n$ be the recurrent sequences defined prior to Theorem 4. Then, the following equality holds for all $n\in\mathbb{N}$ 
\begin{equation}\label{eq4t2}
\varphi(n)= \tilde{\alpha}_n \varphi(1) + \tilde{\delta}_n (4- \mathbb{E} S).
\end{equation}
Indeed, it is evident for $n=1$ and for $n=2$ is is implied by (\ref{eq:one less})
$$
\varphi(2) = -\frac{1}{y_0}\varphi(1)(y_0+y_1) +  \frac{1}{y_0}(4-\E S) = \tilde{\alpha}_2 \varphi(1) + \tilde{\delta}_2 (4-\E S).
$$
For $n \geqslant 3$, it is confirmed by induction and
the rest is evident by setting the difference $\varphi(n+1) - \varphi(n)$ an etc.

It remains to show that we do not divide by zero obtaining an expression of $\varphi(0)$ from eq. (\ref{1x1}) or $\varphi(1)$ from eq. (\ref{1x1_v1}). As Theorem \ref{T_1x1} deals with three underlying scenarios related to where the distribution of $X+Y$ can start not to violate the net profit condition and $s_0=s_1=0, s_2>0$, we have three slightly different types of recurrent sequences to go trough. Under scenario s.1 we have

\begin{align*}
\hat{\alpha}_0=1,\,\hat{\alpha}_1=-\frac{1}{y_1},\,
\hat{\alpha}_n=\frac{1}{s_2}\left(\hat{\alpha}_{n-2}-\sum_{k=1}^{n-1}\hat{\alpha}_k s_{n+2-k}-x_n\right) ,\,n=2,3,\,\ldots,
\end{align*}
while s.2 implies
\begin{align*}
\hat{\alpha}_0=1,\,\hat{\alpha}_1=-\frac{1}{x_0y_2},\,
\hat{\alpha}_n=\frac{1}{s_2}\left(\hat{\alpha}_{n-2}-\sum_{k=1}^{n-1}\hat{\alpha}_k s_{n+2-k}\right) ,\,n=2,3,\,\ldots,
\end{align*}

The proof of $\hat{\alpha}_{n+1}-\hat{\alpha}_{n}\neq0$ for all $n\in\mathbb{N}_0$ under s.1 or s.2 is almost identical to the one given in \cite[p. 937]{DS} or \cite[p. 16]{GS} accordingly.

Under s.3 we have
\begin{align*}
\tilde{\alpha}_1=1,\,\tilde{\alpha}_2=-\frac{y_0+y_1}{y_0},\,
\tilde{\alpha}_n=\frac{1}{s_2}\left(\tilde{\alpha}_{n-2}-\sum_{k=1}^{n-1}\tilde{\alpha}_k s_{n+2-k}+(x_{n+1}-x_n)y_0\right) ,\,n=3,4,\,\ldots
\end{align*}
The property $\tilde{\alpha}_{n+1}-\tilde{\alpha}_{n}\neq0$ for $n\in\mathbb{N}$
is implied by
$1\leqslant\tilde{\alpha}_{2n+1}\leqslant\tilde{\alpha}_{2n+3}$ and
$-1\geqslant\tilde{\alpha}_{2n}\geqslant\tilde{\alpha}_{2n+2}$. To verify the last inequalities we use induction. For $n=1$ we have that
\begin{align*}
\tilde{\alpha}_{4}=\frac{\tilde{\alpha}_{2}-\sum_{k=1}^{3}\tilde{\alpha}_{k}s_{6-k}+(x_5-x_4)y_0}{s_2}\leqslant\frac{\tilde{\alpha}_{2}-\tilde{\alpha}_{2}s_4+x_5y_0}{s_2}
\leqslant\tilde{\alpha}_{2}\leqslant-1
\end{align*}
and
\begin{align*}
\tilde{\alpha}_{3}
=\frac{1-\tilde{\alpha}_{1}s_4-\tilde{\alpha}_{2}s_3+(x_4-x_3)y_0}{s_2}
=\frac{1}{y_0}\left(\frac{1}{x_2}-y_2+y_1+\frac{y_1^2}{y_0}\right)
\geqslant 1=\tilde{\alpha}_{1}.
\end{align*}
For arbitrary $n\in\mathbb{N}$, under induction hypothesis, it follows
\begin{align*}
&\tilde{\alpha}_{2n+2}=\frac{1}{s_2}\left(\tilde{\alpha}_{2n}-\sum_{k=1}^{2n+1}s_{2n+4-k}\tilde{\alpha}_{k}
+(x_{2n+3}-x_{2n+2})y_0
\right)
\\
&=\frac{1}{s_2}\left(\tilde{\alpha}_{2n}-\left(s_3\tilde{\alpha}_{2n+1}+\ldots+s_{2n+3}\tilde{\alpha}_{1}\right)
-\left(s_4\tilde{\alpha}_{2n}+\ldots+s_{2n+2}\tilde{\alpha}_{2}\right)
+(x_{2n+3}-x_{2n+2})y_0\right)
\\
&\leqslant\frac{1}{s_2}\left(\tilde{\alpha}_{2n}-\tilde{\alpha}_1(s_3+\ldots+s_{2n+3})-\tilde{\alpha}_{2n}(s_4+\ldots+s_{2n+2})
+x_{2n+3}y_0\right)\\
&\leqslant\frac{1}{s_2}\left(\tilde{\alpha}_{2n}(1-s_4-\ldots-s_{2n+2})-\tilde{\alpha}_{1}(s_3+\ldots+s_{2n+1})\right)\leqslant\tilde{\alpha}_{2n}
\end{align*}
and
\begin{align*}
&\tilde{\alpha}_{2n+3}=\frac{1}{s_2}\left(\tilde{\alpha}_{2n+1}-\sum_{k=1}^{2n+2}s_{2n+5-k}\tilde{\alpha}_{k}
+(x_{2n+4}-x_{2n+3})y_0\right)
\\
&=\frac{1}{s_2}\left(\tilde{\alpha}_{2n+1}-\left(s_4\tilde{\alpha}_{2n+1}+\ldots+s_{2n+4}\tilde{\alpha}_{1}\right)
-\left(s_3\tilde{\alpha}_{2n+2}+\ldots+s_{2n+3}\tilde{\alpha}_{2}\right)
+(x_{2n+4}-x_{2n+3})y_0\right)
\\
&\geqslant\frac{1}{s_2}\left(\tilde{\alpha}_{2n+1}-\tilde{\alpha}_{2n+1}(s_4+\ldots+s_{2n+4})-x_{2n+3}y_0\right)
\geqslant\tilde{\alpha}_{2n+1}.
\end{align*}
\end{proof}

It is true that we can avoid differences of $\f(n+1)-\f(n)$ and etc. in Theorems \ref{T_3x3}--\ref{T_1x1}. Instead of that, we can utilize expressions (\ref{eq}), (\ref{eq2}), (\ref{eq4t}) accordingly and obtain needed initial values $\f(0),\,\f(1),\,\f(2)$ and etc. based on $\f(n)\approx1$ if $n$ is sufficiently large. However, for some slowly increasing $\f(n)$  
the assumption $\f(n+1)-\f(n)\approx0$ seems to be more accurate than $\f(n)\approx1$. Some thoughts on that are given in \cite[Sec. 5]{GS}. 

\begin{proof}[Proof of Theorem \ref{T_0x0}]
We start with the cases v.2 and v.4. For $u=0$ or $u=1$, the eq. (\ref{main_eq}) accordingly imply
\begin{align*}
\varphi(1) = \frac{\varphi(0)}{s_3 - x_3y_0 - x_2y_1},\,
\varphi(2)= \frac{\varphi(1) (1-s_4+x_4y_0 + x_3y_1)}{s_3-x_3y_0}.
\end{align*}
From this
$$
\varphi(2)= \frac{\varphi(0) (1-s_4+x_4y_0 + x_3y_1)}{(s_3-x_3y_0)(s_3-x_3y_0-x_2y_1)}.
$$
Finally, equality (\ref{eq:one less}) can be rewritten as
$$
\varphi(0) + \varphi(0)\frac{\overline{X}(2)y_0 + \overline{X}(1)y_1}{s_3 - x_3y_0 - x_2y_1} + \varphi(0)\frac{\overline{X}(1)y_0(1-s_4+x_4y_0 + x_3y_1)}{(s_3-x_3y_0)(s_3-x_3y_0-x_2y_1)}=4-\E S.
$$
Therefore
$$
\varphi(0) = \frac{4-\E S}{1+\frac{\overline{X}(2)y_0 + \overline{X}(1)y_1}{s_3 - x_3y_0 - x_2y_1} + \frac{\overline{X}(1)y_0(1-s_4+x_4y_0 + x_3y_1)}{(s_3-x_3y_0)(s_3-x_3y_0-x_2y_1)}}
=\frac{4-\E S}{1+\frac{\overline{X}(1)y_1}{x_1y_2+x_0y_3}}
$$
as $y_0=0$ under v.2 and v.4.

For the next case v.1, if $u=0$, from eq. (\ref{main_eq}) we have
$$
\varphi(0) = \varphi(1)(s_3-x_2y_1) = 0
$$
so the eq. (\ref{eq:one less}) gives
$$
\varphi(1)=\frac{4-\E S}{y_1}.
$$
The result $\varphi(0)=0$ is natural due to $\min X=2$.

For the last case v.3, if $u=0$, from (\ref{main_eq}) we get
$$
\varphi(0) = \varphi(1)(s_3-x_3y_0) = 0
$$
and, if $u=1$,
$$
\varphi(1) = \varphi(1)s_4 +\varphi(2)s_3 -\varphi(1)(x_4y_0+x_3y_1)-\varphi(2)x_3y_0
$$
therefore
$$
\varphi(1)(1-s_4+x_4y_0+x_3y_1) = \varphi(2)(s_3- x_3y_0)=0.
$$
Consequently, eq. (\ref{eq:one less}) implies
$$
\varphi(2) = \frac{4-\E S}{y_0}.
$$
Also, $\varphi(0)=\varphi(1)=0$ is because of $\min X=3$.

The remaining survival probabilities $\f(u),\,u\geqslant2$, are implied by (\ref{main_eq}). 
\end{proof}

\begin{proof}[Proof of Theorem \ref{no net}]
The first case, $\varphi(u)=0$ for all $u\in\mathbb{N}_0$ if $\mathbb{E}S>4$, is implied by combining eq. (\ref{eq:before_l}) with (\ref{expectation}), which gives
\begin{align}\label{eq:impl_zero}
\f(0)+\left(\overline{X}(2)y_0+\overline{X}(1)y_1\right)\f(1)+\overline{X}(1)y_0\f(2)+\sum_{k=1}^{3}\f(k)S(3-k)=(4-\E S)\cdot\varphi(\infty). 
\end{align}

As the left hand side of (\ref{eq:impl_zero}) is non-negative and $4-\mathbb{E}S<0$, then $\varphi(\infty)=0$, and consequently $\varphi(u)=0$ for all $u\in\mathbb{N}_0$ as $\varphi$ is non-decreasing.

For the second case, $\varphi(u)=0$ for all $u\in\mathbb{N}_0$ if $\mathbb{E}S=4$ and $s_4<1$, we observe that (\ref{eq:impl_zero}) becomes
\begin{align}\label{eq:impl_zero_v1}
\f(0)+(x_0y_2+y_0+y_1)\f(1)+(x_0y_1+y_0)\f(2)+x_0y_0\f(3)=0. 
\end{align}
The rest for this second case is concluded in the same manner as Theorems \ref{T_3x3}--\ref{T_0x0} are structured. Indeed, if $s_0>0$ as in Theorem \ref{T_3x3}, then eq. (\ref{eq:impl_zero_v1}) implies $\f(0)=\ldots=\f(3)=0$ and $\f(u)=0$ for $u\geqslant4$ follows by eq. (\ref{main_eq}). If $s_0=0$ and $s_1>0$ (as in Theorem \ref{T_2x2}), which consists from two underlying cases $x_0=0,\,y_0>0,\,x_1>0$ or $x_0>0,\,y_0=0,\,y_1>0$, then (\ref{eq:impl_zero_v1}) implies $\f(0)=\ldots=\f(2)=0$ and eq. (\ref{main_eq}) does the rest. Two remaining dependencies $s_0=s_1=0,\,s_2>0$ and $s_0=s_1=s_2=0,\,s_3>0$ follow by the same arguments.

The third and last case on $\varphi$ expressions when $\mathbb{E}S=4$ and $s_4=x_4y_0+x_3y_1+x_2y_2+x_1y_3+x_0y_4=1$ is implied by the definition of the bi-seasonal discrete time risk model with income rate two (see (\ref{model})). Then the modeled function $W(t)$ equals:
\begin{itemize}
    \item $W(t)=u-2\cdot\mathbbm{1}_{\{t\text{ is odd}\}}$ if $x_4=y_0=1$, 
    \item $W(t)=u-1\cdot\mathbbm{1}_{\{t\text{ is odd}\}}$ if $x_3=y_1=1$,
    \item $W(t)=u$ if $x_2=y_2=1$,
    \item $W(t)=u+1\cdot\mathbbm{1}_{\{t\text{ is odd}\}}$ if $x_1=y_3=1$,
    \item $W(t)=u+2\cdot\mathbbm{1}_{\{t\text{ is odd}\}}$ if $x_0=y_4=1$,
\end{itemize}
The proof follows identifying when $W(t)>0$ for all $t=1,2,\,\ldots$

\bigskip

Similar thoughts as given in this proof appear in \cite[p. 935]{DS}.
\end{proof}

\section{Numerical examples}\label{sec:calculations}
In this section, using program \cite{Mathematica}, we demonstrate numerical outputs of Theorems from Section \ref{sec:Statements} assuming that r.vs., generating the bi-seasonal discrete time risk model with income rate two, follow the {\it displaced Poisson distribution} $\mathcal{P}(\lambda,\xi)$ with parameters $\lambda>0$ and $\xi\in\mathbb{N}_0$, which PDF is
\begin{align*}
\T(X=m)=e^{-\lambda}\frac{\lambda^{m-\xi}}{(m-\xi)!},\,
m=\xi,\,\xi+1,\,\dots
\end{align*}
It can be verified that $\E X=\lambda+\xi$ and $X+Y\sim\mathcal{P}(\lambda_1+\lambda_2,\xi_1+\xi_2)$ if 
$X\sim\mathcal{P}(\lambda_1,\xi_1)$, 
$Y\sim\mathcal{P}(\lambda_2,\xi_2)$ and $X$, $Y$ are independent. See \cite{Staff} for more information on displaced Poisson distribution. 

All of the output tables of $\varphi(u,T)$ and $\varphi(u)$ below are structured as follows: the present survival probabilities are rounded up to three decimal places except when the numbers are 0 or 1; parameters $T$ and $u$ are chosen to reflect changes of survival probabilities; the size of $n=150$ is considered as high enough to reach a sufficient accuracy when Theorems \ref{T_3x3}--\ref{T_1x1} are employed to find a needed initial values of $\f$.   

\begin{ex}\label{ex1}
Let $X\sim\mathcal{P}(1,0)$ and $Y\sim\mathcal{P}(2,0)$. Using Theorem \ref{finite} and Theorem \ref{T_3x3} we obtain the following table
\end{ex}
\begin{table}[H]
\centering
\caption{Finite and ultimate time survival probabilities with r.vs. from Example \ref{ex1}.}\label{tab2}
\begin{tabular}{ccccccccc}
\toprule
\boldmath{$T\backslash u$}&\boldmath{$0$}&\boldmath{$1$}&\boldmath{$2$}&\boldmath{$3$}&\boldmath{$4$}&\boldmath{$5$}&\boldmath{$10$}&\boldmath{$15$}\\
\midrule
$1$&$0.736$&$0.920$&$0.981$&$0.996$&$0.999$&$1$&$1$&$1$\\
$2$&$0.564$&$0.788$&$0.909$&$0.965$&$0.988$&$0.996$&$1$&$1$\\
$3$&$0.547$&$0.771$&$0.898$&$0.959$&$0.985$&$0.995$&$1$&$1$\\
$4$&$0.505$&$0.727$&$0.863$&$0.936$&$0.972$&$0.989$&$1$&$1$\\
$5$&$0.499$&$0.720$&$0.857$&$0.932$&$0.969$&$0.987$&$1$&$1$\\
$10$&$0.46$&$0.673$&$0.813$&$0.898$&$0.946$&$0.972$&$0.999$&$1$\\
$20$&$0.446$&$0.656$&$0.795$&$0.882$&$0.933$&$0.962$&$0.998$&$1$\\
$30$&$0.443$&$0.652$&$0.791$&$0.878$&$0.930$&$0.960$&$0.998$&$1$\\
$40$&$0.443$&$0.651$&$0.790$&$0.877$&$0.929$&$0.959$&$0.997$&$1$\\
$50$&$0.442$&$0.650$&$0.790$&$0.876$&$0.928$&$0.959$&$0.997$&$1$\\
$\infty$&$0.442$&$0.650$&$0.790$&$0.876$&$0.928$&$0.958$&$0.997$&$1$\\
\bottomrule
\end{tabular}
\end{table}

\begin{ex}\label{ex2}
Let $X\sim\mathcal{P}(1,1)$ and $Y\sim\mathcal{P}(19/10,0)$. Using Theorem \ref{finite} and Theorem \ref{T_2x2} we obtain the following table
\end{ex}
\begin{table}[H]
\centering
\caption{Finite and ultimate time survival probabilities with r.vs. from Example \ref{ex2}.}\label{tab3}
\begin{tabular}{ccccccccccc}
\toprule
\boldmath{$T\backslash u$}&\boldmath{$0$}&\boldmath{$1$}&\boldmath{$2$}&\boldmath{$3$}&\boldmath{$4$}&\boldmath{$5$}&\boldmath{$10$}&\boldmath{$20$}&\boldmath{$30$}&\boldmath{$40$}\\
\midrule
$1$&$0.368$&$0.736$&$0.920$&$0.981$&$0.996$&$0.999$&$1$&$1$&$1$&$1$\\
$2$&$0.259$&$0.581$&$0.803$&$0.919$&$0.970$&$0.990$&$1$&$1$&$1$&$1$\\
$3$&$0.223$&$0.518$&$0.743$&$0.877$&$0.947$&$0.979$&$1$&$1$&$1$&$1$\\
$4$&$0.192$&$0.458$&$0.677$&$0.823$&$0.910$&$0.957$&$1$&$1$&$1$&$1$\\
$5$&$0.177$&$0.428$&$0.641$&$0.791$&$0.886$&$0.942$&$0.999$&$1$&$1$&$1$\\
$10$&$0.130$&$0.324$&$0.505$&$0.652$&$0.765$&$0.847$&$0.990$&$1$&$1$&$1$\\
$20$&$0.098$&$0.248$&$0.396$&$0.525$&$0.634$&$0.724$&$0.951$&$1$&$1$&$1$\\
$30$&$0.084$&$0.214$&$0.343$&$0.460$&$0.562$&$0.649$&$0.908$&$0.998$&$1$&$1$\\
$40$&$0.076$&$0.193$&$0.311$&$0.419$&$0.515$&$0.599$&$0.869$&$0.994$&$1$&$1$\\
$50$&$0.070$&$0.179$&$0.289$&$0.390$&$0.481$&$0.562$&$0.837$&$0.989$&$1$&$1$\\
$100$&$0.057$&$0.144$&$0.234$&$0.318$&$0.395$&$0.465$&$0.731$&$0.952$&$0.995$&$1$\\
$\infty$&$0.037$&$0.094$&$0.152$&$0.208$&$0.259$&$0.307$&$0.506$&$0.748$&$0.872$&$0.935$\\
\bottomrule
\end{tabular}
\end{table}
\begin{ex}\label{ex3}
Let $X\sim\mathcal{P}(1,1)$ and $Y\sim\mathcal{P}(9/10,1)$. Using Theorem \ref{finite} and Theorem \ref{T_1x1} we obtain the following table
\end{ex}

\begin{table}[H]
\centering
\caption{Finite and ultimate time survival probabilities with r.vs. from Example \ref{ex3}.}\label{tab4}
\begin{tabular}{ccccccccccc}
\toprule
\boldmath{$T\backslash u$}&\boldmath{$0$}&\boldmath{$1$}&\boldmath{$2$}&\boldmath{$3$}&\boldmath{$4$}&\boldmath{$5$}&\boldmath{$10$}&\boldmath{$20$}&\boldmath{$30$}&\boldmath{$40$}\\
\midrule
$1$&$0.368$&$0.736$&$0.920$&$0.981$&$0.996$&$0.999$&$1$&$1$&$1$&$1$\\
$2$&$0.284$&$0.629$&$0.850$&$0.950$&$0.986$&$0.996$&$1$&$1$&$1$&$1$\\
$3$&$0.237$&$0.552$&$0.784$&$0.910$&$0.967$&$0.989$&$1$&$1$&$1$&$1$\\
$4$&$0.212$&$0.506$&$0.739$&$0.878$&$0.949$&$0.980$&$1$&$1$&$1$&$1$\\
$5$&$0.191$&$0.466$&$0.695$&$0.844$&$0.927$&$0.968$&$1$&$1$&$1$&$1$\\
$10$&$0.145$&$0.366$&$0.572$&$0.729$&$0.837$&$0.908$&$0.997$&$1$&$1$&$1$\\
$20$&$0.111$&$0.286$&$0.459$&$0.605$&$0.720$&$0.807$&$0.981$&$1$&$1$&$1$\\
$30$&$0.096$&$0.249$&$0.404$&$0.538$&$0.650$&$0.740$&$0.957$&$1$&$1$&$1$\\
$40$&$0.087$&$0.227$&$0.370$&$0.496$&$0.604$&$0.693$&$0.933$&$0.999$&$1$&$1$\\
$50$&$0.081$&$0.212$&$0.346$&$0.466$&$0.570$&$0.658$&$0.910$&$0.998$&$1$&$1$\\
$100$&$0.067$&$0.175$&$0.288$&$0.390$&$0.482$&$0.562$&$0.828$&$0.984$&$0.999$&$1$\\
$\infty$&$0.048$&$0.127$&$0.209$&$0.286$&$0.355$&$0.417$&$0.649$&$0.873$&$0.954$&$0.983$\\
\bottomrule
\end{tabular}
\end{table}

\begin{ex}\label{ex4}
Let $X\sim \mathcal{P}(1/2,2)$ and $Y\sim \mathcal{P}(1/3,1)$. Using Theorem \ref{finite} and Theorem \ref{T_0x0} we obtain the following table
\end{ex}

\begin{table}[H]
\centering
\caption{Finite and ultimate time survival probabilities with r.vs. from Example \ref{ex4}.}\label{tab5}
\begin{tabular}{ccccccccccc}
\toprule
\boldmath{$T\backslash u$}&\boldmath{$0$}&\boldmath{$1$}&\boldmath{$2$}&\boldmath{$3$}&\boldmath{$4$}&\boldmath{$5$}&\boldmath{$10$}&\boldmath{$15$}&\boldmath{$20$}&\boldmath{$25$}\\
\midrule
$1$&$0.607$&$0.910$&$0.986$&$0.998$&$1$&$1$&$1$&$1$&$1$&$1$\\
$2$&$0.435$&$0.797$&$0.948$&$0.990$&$0.998$&$1$&$1$&$1$&$1$&$1$\\
$3$&$0.395$&$0.758$&$0.928$&$0.983$&$0.997$&$0.999$&$1$&$1$&$1$&$1$\\
$4$&$0.346$&$0.700$&$0.894$&$0.969$&$0.992$&$0.998$&$1$&$1$&$1$&$1$\\
$5$&$0.329$&$0.678$&$0.878$&$0.961$&$0.989$&$0.997$&$1$&$1$&$1$&$1$\\
$10$&$0.262$&$0.573$&$0.788$&$0.906$&$0.962$&$0.986$&$1$&$1$&$1$&$1$\\
$20$&$0.219$&$0.494$&$0.705$&$0.838$&$0.916$&$0.959$&$0.999$&$1$&$1$&$1$\\
$30$&$0.202$&$0.459$&$0.662$&$0.799$&$0.884$&$0.936$&$0.998$&$1$&$1$&$1$\\
$40$&$0.192$&$0.439$&$0.637$&$0.773$&$0.862$&$0.918$&$0.996$&$1$&$1$&$1$\\
$50$&$0.186$&$0.426$&$0.620$&$0.756$&$0.846$&$0.905$&$0.994$&$1$&$1$&$1$\\
$100$&$0.173$&$0.398$&$0.583$&$0.716$&$0.808$&$0.872$&$0.985$&$0.999$&$1$&$1$\\
$\infty$&$0.167$&$0.383$&$0.563$&$0.693$&$0.784$&$0.849$&$0.974$&$0.996$&$0.999$&$1$\\
\bottomrule
\end{tabular}
\end{table}

\begin{ex}\label{ex5}
Let $X\sim \mathcal{P}(2,1)$ and $Y\sim \mathcal{P}(1,1)$. Note that $\mathbb{E}S=5>4$ in the case under consideration. Using Theorem \ref{finite} and Theorem \ref{no net} we obtain the following table
\end{ex}

\begin{table}[H]
\centering
\caption{Finite time survival probabilities with r.vs. from Example \ref{ex5}.}\label{tab6}
\begin{tabular}{cccccccccccc}
\toprule
\boldmath{$T\backslash u$}&\boldmath{$0$}&\boldmath{$1$}&\boldmath{$2$}&\boldmath{$3$}&\boldmath{$4$}&\boldmath{$5$}&\boldmath{$10$}&\boldmath{$20$}&\boldmath{$30$}&\boldmath{$40$}&\boldmath{$50$}\\
\midrule
$1$&$0.135$&$0.406$&$0.677$&$0.857$&$0.947$&$0.983$&$1$&$1$&$1$&$1$&$1$\\
$2$&$0.100$&$0.324$&$0.581$&$0.782$&$0.903$&$0.962$&$1$&$1$&$1$&$1$&$1$\\
$3$&$0.054$&$0.194$&$0.391$&$0.589$&$0.750$&$0.862$&$0.998$&$1$&$1$&$1$&$1$\\
$4$&$0.045$&$0.166$&$0.343$&$0.532$&$0.696$&$0.820$&$0.996$&$1$&$1$&$1$&$1$\\
$5$&$0.029$&$0.112$&$0.243$&$0.401$&$0.558$&$0.696$&$0.982$&$1$&$1$&$1$&$1$\\
$10$&$0.010$&$0.042$&$0.099$&$0.179$&$0.278$&$0.388$&$0.854$&$1$&$1$&$1$&$1$\\
$20$&$0.002$&$0.008$&$0.020$&$0.038$&$0.065$&$0.102$&$0.417$&$0.946$&$0.999$&$1$&$1$\\
$30$&$0$&$0.002$&$0.005$&$0.010$&$0.017$&$0.029$&$0.161$&$0.723$&$0.978$&$1$&$1$\\
$40$&$0$&$0.001$&$0.001$&$0.003$&$0.005$&$0.008$&$0.058$&$0.439$&$0.874$&$0.991$&$1$\\
$50$&$0$&$0$&$0$&$0.001$&$0.002$&$0.003$&$0.020 $&$0.228$&$0.674$&$0.943$&$0.996$\\
$100$&$0$&$0$&$0$&$0$&$0$&$0$&$0$&$0.003$&$0.035$&$0.173$&$0.461$\\
$\infty$&$0$&$0$&$0$&$0$&$0$&$0$&$0$&$0$&$0$&$0$&$0$\\
\bottomrule
\end{tabular}
\end{table}

As a general overview on the results of survival probabilities present in the above Examples \ref{ex1}--\ref{ex5}, it can be commented that the difference of $\mathbb{E}S$ to $4$ makes a high impact on the likelihood of survival. Example \ref{ex1} shows that expenses, represented as random claims, which are not too harsh on average ($\mathbb{E}S=3$), may be well covered by the initial surplus $u$ and guaranteed survival is reached when $u=15$. On the other hand, Table \ref{ex5} illustrates that a quite confident short term survival possibility can be achieved with a sufficient level of initial savings even when an occurring random expenses are more "aggressive".   

\section{Acknowledgements}
We want to thank anonymous referee for reviewing the manuscript. We also feel very appreciative to professor Jonas Šiaulys for his support and advises writing the paper.    

\bibliography{mybibfile}

\end{document}